\documentclass[11pt]{article}
\usepackage{mathtools}
\usepackage{ctable}
\usepackage{amsmath}
\usepackage{amsfonts}
\usepackage{latexsym}
\usepackage{amsthm}
\usepackage{color}

\newtheorem{theorem}{Theorem}[section]

\newtheorem{lemma}{Lemma}[section]

\theoremstyle{definition}
\newtheorem{definition}{Definition}
\newtheorem{remark}{Remark}
\newtheorem{example}{Example}
\title{\huge \bf On extropy of past lifetime distribution\\ 
}
\author{
{\sc Osman Kamari}\\
 University of Human Development\\
 Sulaymaniyah, Iraq\\
\and {\sc Francesco Buono}\\
Universit\`a di Napoli Federico II\\
 Italy\\
}
\begin{document}
\date{\today}
\maketitle
\begin{abstract}
Recently Qiu et al. (2017) have introduced residual extropy as measure
of uncertainty in residual lifetime distributions analogues to residual
entropy (1996). Also, they obtained some properties and applications of that. In this
paper, we study the extropy to measure the uncertainty in a past lifetime
distribution. This measure of uncertainty is called past extropy. Also it is showed a characterization
result about the past extropy of largest order statistics.
\end{abstract}

\medskip\noindent
{\em Keywords:\/} 
Reversed residual lifetime, Past extropy, Characterization, Order statistics.

\medskip\noindent
{\em AMS Subject Classification}: 94A17, 62B10, 62G30

\section{Introduction}

The concept of Shannon entropy as a seminal measure of uncertainty for a random
variable was proposed by Shannon (1948). Shannon entropy $H(f)$ for a
non-negative and absolutely continuous random variable $X$ is defined as
follows:

\begin{equation}
H\left(f\right)=-\mathbb{E}[\log f(x)]=-\int_0^{+\infty}f\left(x\right)\log f\left(x\right)\ \mathrm dx,
\end{equation}
where F and f are cumulative distribution function (CDF) and probability density
function (pdf), respectively. There are huge literatures devoted to the
applications, generalizations and properties of Shannon entropy (see, e.g.
Cover and Thomas, 2006).

Recently, a new measure of uncertainty was proposed by Lad et al. (2015) called
extyopy as a complement dual of Shannon entropy (1948). For a  non-negative random
variable $X$ the extropy is defined as below:

\begin{equation}
J\left(X\right)=-\frac{1}{2}\int_0^{+\infty}f^2(x) \mathrm dx.
\end{equation}
It's obvious that $J\left(X\right)\leq 0.$

One of the statistical applications of extropy is to score the forecasting
distributions using the total log scoring rule.

Study of duration is an interest subject in many fields of science such as
reliability, survival analysis, and forensic science. In these areas, the
additional life time given that the component or system or a living organism has
survived up to time $t$ is termed the residual life function of the component. If $X$
is the life of a component, then  $X_t=\left(X-t|X>t\right)$  is called the
residual life function. If a component is known to have survived to age $t$ 
then extropy is no longer useful to measure the uncertainty of remaining 
lifetime of the component.

Therefore, Ebrahimi (1996) defined the entropy for resudial lifetime
$X_t=\left(X-t|X>t\right)$ as a dynamic form of uncertainty called the residual
entropy at time $t$ and defined as

\begin{equation*}
H(X;t)=-\int_t^{+\infty}\frac{f(x)}{\overline F (t)}\log \frac{f(x)}{\overline F(t)}\ \mathrm dx,
\end{equation*}

where $\overline F(t)=\mathbb P(X>t)=1-F(t)$ is the survival (reliability) funltion of $X$.

Analogous to residual entropy, Qiu et al. (2017) defined the extropy for
residual lifetime $X_t$ called the residual extropy at time $t$ and defined as
\begin{equation}
J\left(X_t\right)=-\frac{1}{2}\int_0^{+\infty}f_{X_t}^2(x) \mathrm dx=-\frac{1}{2\overline F^2(t)}\int_t^{+\infty}f^2(x) \mathrm dx.
\end{equation}

In many situations, uncertainty can relate to the past. Suppose the random
variable $X$ is the lifetime of a component, system or a living organism, having an
absolutely continuous distribution function $F_X(t)$ and the density function $f_X(t)$. For
$t>0$, let the random variable $_{t}X=(t-X|X<t)$ be the time elapsed after failure till time $t$,
given that the component has already failed at time $t$. We denote the random
variable $_{t}X$, the reversed residual life (past lifetime). For instance, at time $t$,
one has under gone a medical test to check for a certain disease. Suppose that
the test result is positive. If $X$ is the age when the patient was infected, then
it is known that $X<t$. Now the question is, how much time has elapsed since the
patient has been infected by this disease?  Based on this idea, Di Crescenzo and Longobardi (2002) introduced
the entropy of the reversed residual lifetime $_{t}X$ as a dynamic measure of
uncertainty called past entropy as follows:
\begin{equation*}
H\left(X;[t]\right)=-\int_0^t \frac{f(x)}{F (t)}\log \frac{f(x)}{F(t)}\ \mathrm dx.
\end{equation*}
This measure is dual of residual entropy introduced by Ebrahimi (1996).

In this paper, we study the extropy for $_{t}X$ as dual of residual extropy that is
called past extropy and it is defined as below (see also Krishnan et al. (2020)):
\begin{equation}
\label{eq4}
J\left(_{t}X\right)=-\frac{1}{2}\int_0^{+\infty}f_{_{t}X}^2(x) \mathrm dx=-\frac{1}{2 F^2(t)}\int_0^{t}f^2(x) \mathrm dx,
\end{equation}
where $f_{_{t}X}(x)=\frac{f(t-x)}{F(t)}$, for $x\in(0,t)$. It can be seen that for $t\ge0$, $J\left(_{t}X\right)$ possesses all the properties of $J(X)$.

\begin{remark}
It's clear that $J\left(_{+\infty}X\right)=J(X)$.
\end{remark}

Past extropy has applications in the context of information theory, reliability
and survival analysis, insurance, forensic science and other related fields
beceuse in that a lifetime distribution truncated above is of utmost importance.

The paper is organized as follows: in section 2, an approach to measure of
uncertainty in the past lifetime distribution is proposed. Then it is studied a characterization result with the reversed failure rate.
Following a characterization result is given based on past extropy of
the largest order statistics in section 3.

\section{Past extropy and some characterizations}

Analogous to residual extropy (Qiu et al. (2017)), the extropy for $_{t}X$ is called past extropy and for a non-negative random variable $X$ is as below:

\begin{equation}
\label{eq1}
J\left(_{t}X\right)=-\frac{1}{2} \int_0^{+\infty}f_{_{t}X}^2(x) \mathrm dx=-\frac{1}{2F^2 (t)} \int_0^t f^2 (x)\mathrm dx,
\end{equation}   
where $f_{_{t}X}(x)=\frac{f(t-x)}{F(t)}$, for $x\in(0,t)$ is the density function of $_{t}X$. It’s clear that $J({_t}X )\leq0$ while the residual entropy of a continuous distribution may take any value on $[-\infty,+\infty]$. Also, $J\left(_{+\infty}X\right)=J(X)$.

\begin{example}
\label{ex1}
\begin{itemize}
\item[a)] If $X\sim Exp(\lambda)$, then $J\left(_{t}X\right)=-\frac{\lambda}{4}\frac{1+\mathrm e^{-\lambda t}}{1-\mathrm e^{-\lambda t}}$ for $t>0$.
This shows that the past extropy of exponential distribution is an increasing function of t.
\item[b)] If $X\sim U(0,b)$, then $J\left(_{t}X\right)=-\frac{1}{2t}$.
\item[c)] If $X$ has power distribution with parameter $\alpha>0$, i.e. $f(x)=\alpha x^{(\alpha-1)}$, $0<x<1$, then $J\left(_{t}X\right)=\frac{-\alpha^2}{2(2\alpha-1)t}$.
\item[d)] If $X$ has Pareto distribution with parameters $\theta>0, x_0>0$, i.e. $f(x)=\frac{\theta}{x_0}\frac{x_0^{\theta+1}}{x^{\theta+1}}$, $x>x_0$, then $J\left(_{t}X\right)=\frac{\theta^2}{2(2\theta+1)(t^{\theta}-x_0^{\theta})^2}\left[\frac{x_0^{2\theta}}{t}-\frac{t^{2\theta}}{x_0}\right]$.
\end{itemize}
\end{example}

There is a functional relation between past extropy and residual extropy as follows:
\begin{equation*}
J(X)=F^2(t)J\left(_{t}X\right)+\overline F^2(t)J\left(X_t\right), \forall t>0.
\end{equation*}
In fact
\begin{eqnarray*}
F^2(t)J\left(_{t}X\right)+\overline F^2(t)J\left(X_t\right)&=&-\frac{1}{2}\int_t^{+\infty}f^2(x) \mathrm dx-\frac{1}{2}\int_0^{t}f^2(x) \mathrm dx\\
&=&-\frac{1}{2}\int_0^{+\infty}f^2(x) \mathrm dx=J(X).
\end{eqnarray*}


From~\eqref{eq4} we can rewrite the following expression for the past extropy:
\begin{equation*}
J\left(_{t}X\right)=\frac{-\tau^2(t)}{2f^2(t)}\int_0^{t}f^2(x) \mathrm dx,
\end{equation*}
where $\tau(t)=\frac{f(t)}{F(t)}$ is the reversed failure rate.

\begin{definition}
A random variable is said to be increasing (decreasing) in past extropy if $J\left(_{t}X\right)$ is an increasing (decreasing) function of $t$.
\end{definition}

\begin{theorem}
$J\left(_{t}X\right)$ is increasing (decreasing) if and only if  $J\left(_{t}X\right)\leq(\ge)\frac{-1}{4}\tau(t)$.
\end{theorem}

\begin{proof}
From~\eqref{eq1} we get
\begin{equation*}
\frac{\mathrm d}{\mathrm dt}J\left(_{t}X\right)=-2\tau(t)J\left(_{t}X\right)-\frac{1}{2}\tau^2(t).
\end{equation*}
Then $J\left(_{t}X\right)$ is increasing if and only if
\begin{equation*}
2\tau(t)J\left(_{t}X\right)+\frac{1}{2}\tau^2(t)\leq0,
\end{equation*}
but $\tau(t)\ge0$ so this is equivalent to
\begin{equation*}
J\left(_{t}X\right)\leq-\frac{1}{4}\tau(t).
\end{equation*}
\end{proof}

\begin{theorem}
The past extropy $J\left(_{t}X\right)$ of $X$ is uniquely determined by $\tau(t)$.
\end{theorem}

\begin{proof}
From~\eqref{eq1} we get
\begin{equation*}
\frac{\mathrm d}{\mathrm dt}J\left(_{t}X\right)=-2\tau(t)J\left(_{t}X\right)-\frac{1}{2}\tau^2(t).
\end{equation*}
So we have a linear differential equation of order one and it can be solved in the following way
\begin{equation*}
J\left(_{t}X\right)=\mathrm e^{-2\int_{t_0}^t \tau(s) \mathrm ds} \left[J\left(_{t_0}X\right) -\int_{t_0}^t \frac{1}{2}\tau^2(s)\mathrm e^{2\int_{t_0}^s \tau(y) \mathrm dy}\mathrm ds \right],
\end{equation*}
where we can use the boundary condition $J\left(_{+\infty}X\right)=J(X)$, so we get
\begin{equation}
\label{eq2}
J\left(_{t}X\right)=\mathrm e^{2\int_{t}^{+\infty} \tau(s) \mathrm ds} \left[J\left(X\right) +\int_{t}^{+\infty} \frac{1}{2}\tau^2(s)\mathrm e^{-2\int_{s}^{+\infty} \tau(y) \mathrm dy}\mathrm ds \right].
\end{equation}
\end{proof}

\begin{example}
Let $X\sim Exp(\lambda)$, with reversed failure rate $\tau(t)=\frac{\lambda \mathrm e^{-\lambda t}}{1-\mathrm e^{-\lambda t}}$. It follows from~\eqref{eq2} that
\begin{eqnarray*}
J\left(_{t}X\right)&=&\mathrm e^{2\int_{t}^{+\infty} \frac{\lambda \mathrm e^{-\lambda s}}{1-\mathrm e^{-\lambda s}} \mathrm ds} \left[J\left(X\right) +\int_{t}^{+\infty} \frac{1}{2}\frac{\lambda^2 \mathrm e^{-2\lambda s}}{\left(1-\mathrm e^{-\lambda s}\right)^2}\mathrm e^{-2\int_{s}^{+\infty} \frac{\lambda \mathrm e^{-\lambda y}}{1-\mathrm e^{-\lambda y}} \mathrm dy}\mathrm ds \right]\\
&=& \left(1-\mathrm e^{-\lambda t}\right)^{-2}\left[J(X)+\frac{1}{2}\int_t^{+\infty} \lambda^2\mathrm e^{-2\lambda s}\mathrm ds\right]\\
&=& \frac{\lambda}{4}\frac{\mathrm e^{-2\lambda t}-1}{\left(1-\mathrm e^{-\lambda t}\right)^2}=-\frac{\lambda}{4}\frac{1+\mathrm e^{-\lambda t}}{1-\mathrm e^{-\lambda t}},
\end{eqnarray*}
so we find again the same result of example~\ref{ex1}.
\end{example}

Using the following definition (see Shaked and Shanthikumar, 2007), we show that $J\left(_{t}X\right)$ is increasing in $t$.

\begin{definition}
Let $X$ and $Y$ be two non-negative variables with reliability functions $\overline F, \overline G$ and pdfs $f, g$ respectively. $X$ is smaller than $Y$
\begin{itemize}
\item[a)] in the likelihood ratio order, denoted by $X\leq_{lr}Y$, if $\frac{f(x)}{g(x)}$ is decreasing in $x\ge0$;
\item[b)] in the usual stochastic order, denoted by $X\leq_{st}Y$ if $\overline F(x)\leq\overline G(x)$ for $x\ge0$.
\end{itemize}
\end{definition}

\begin{remark}
It is well known that if $X\leq_{lr}Y$ then $X\leq_{st}Y$ and $X\leq_{st}Y$ if and only if $\mathbb E(\varphi(Y))\leq(\ge)\mathbb E(\varphi(X))$ for any decreasing (increasing) function $\varphi$.
\end{remark}

\begin{theorem}
\label{thm1}
Let $X$ be a random variable with CDF $F$ and pdf $f$. If $f\left(F^{-1}(x)\right)$ is decreasing in $x\ge0$, then $J\left(_{t}X\right)$ is increasing in $t\ge0$.
\end{theorem}

\begin{proof}
Let $U_t$ be a random variable with uniform distribution on $(0,F(t))$ with pdf $g_t(x)=\frac{1}{F(t)}$ for $x\in(0,F(t))$, then based on~\eqref{eq4} we have
\begin{eqnarray*}
J\left(_{t}X\right)&=&-\frac{1}{2F^2(t)} \int_0^{F(t)}f\left(F^{-1}(u)\right) \mathrm du=-\frac{1}{2F(t)} \int_0^{F(t)} g_t(u)f\left(F^{-1}(u)\right) \mathrm du\\
&=&-\frac{1}{2F(t)}\mathbb E\left[f\left(F^{-1}(U_t)\right) \right].
\end{eqnarray*}

Let $0\leq t_1\leq t_2$. If $0<x\leq F(t_1)$, then $\frac{g_{t_1}(x)}{g_{t_2}(x)}=\frac{F(t_2)}{F(t_1)}$ is a non-negative constant. If $F(t_1)<x\leq F(t_2)$, then $\frac{g_{t_1}(x)}{g_{t_2}(x)}=0$. Therefore $\frac{g_{t_1}(x)}{g_{t_2}(x)}$ is decreasing in $x\in(0,F(t_2))$, which implies $U_{t_1}\leq_{lr}U_{t_2}$. Hence $U_{t_1}\leq_{st}U_{t_2}$ and so
\begin{equation*}
0\leq\mathbb E\left[f\left(F^{-1}(U_{t_2})\right) \right]\leq\mathbb E\left[f\left(F^{-1}(U_{t_1}\right) \right]
\end{equation*}
using the assumption that $f\left(F^{-1}(U_t)\right)$ is a decreasing function. Since $0\leq\frac{1}{F(t_2)}\leq\frac{1}{F(t_1)}$ then
\begin{equation*}
J\left(_{t_1}X\right)=-\frac{1}{2F(t_1)}\mathbb E\left[f\left(F^{-1}(U_{t_1})\right) \right]\leq -\frac{1}{2F(t_2)}\mathbb E\left[f\left(F^{-1}(U_{t_2})\right) \right]=J\left(_{t_2}X\right).
\end{equation*}
\end{proof}

\begin{remark}
Let $X$ be a random variable with CDF $F(x)=x^2$, for $x\in(0,1)$. Then $f\left(F^{-1}(x)\right)=2\sqrt x$ is increasing in $x\in(0,1)$. However $J\left(_{t}X\right)=-\frac{2}{3t}$ is increasing in $t\in(0,1)$. So the condition in theorem~\ref{thm1} that $f\left(F^{-1}(x)\right)$ is decreasing in $x$ is sufficient but not necessary.
\end{remark}

\section{Past extropy of order statistics}

Let $X_1,X_2,\dots,X_n$ be a random sample with distribution function $F$, the order statistics of the sample are defined by the arrangement $X_1,X_2,\dots,X_n$ from the minimum to the maximum by $X_{(1)},X_{(2)},\dots,X_{(n)}$.
Qiu and Jia (2017) defined the residual extropy of the $i-th$ order statistics and showed that the residual extropy of order statistics can determine the underlying distribution uniquely.
Let $X_1,X_2,\dots,X_n$ be continuous and i.i.d. random variables with CDF $F$ indicate the lifetimes of $n$ components of a parallel system. Also $X_{1:n},X_{2:n},\dots,X_{n:n}$ be the ordered lifetimes of the components. Then $X_{n:n}$ represents the lifetime of parallel system with CDF $F_{X_{n:n}} (x)=(F(x))^n$, $x>0$. The CDF of $\left[t-X_{n:n}|X_{n:n}<t\right]$ is $1-\left(\frac{F(t-x)}{F(t)}\right)^n$ where $\left[t-X_{n:n}|X_{n:n}<t\right]$ is called reversed residual lifetime of the system. Now past  extropy for reversed residual lifetime of parallel system with distribution function $F_{X_{n:n}}(x)$ is as follows:
\begin{equation*}
J\left(_{t}X_{n:n}\right)=-\frac{n^2}{2(F(t))^{2n}} \int_0^t f^2 (x)[F(x)]^{2n-2}\mathrm dx.
\end{equation*}

\begin{theorem}
\label{thm7}
If $X$ has an increasing pdf $f$ on $[0,T]$, with $T>t$, then $J\left(_{t}X_{n:n}\right)$ is decreasing in $n\ge1$.
\end{theorem}

\begin{proof}
The pdf of $(X_{n:n}|X_{n:n}\leq t)$ can be expressed as
\begin{equation*}
g_{n:n}^t(x)=\frac{nf(x)F^{n-1}(x)}{F^n(t)}, \mbox{   } x\leq t.
\end{equation*}
We note that
\begin{equation*}
\frac{g_{2n-1:2n-1}^t(x)}{g_{2n+1:2n+1}^t(x)}=\frac{2n-1}{2n+1}\frac{F^2(t)}{F^2(x)}
\end{equation*}
is decreasing in $x\in[0,t]$ and so $(X_{2n-1:2n-1}|X_{2n-1:2n-1}\leq t)\leq_{lr}(X_{2n+1:2n+1}|X_{2n+1:2n+1}\leq t)$ which implies $(X_{2n-1:2n-1}|X_{2n-1:2n-1}\leq t)\leq_{st}(X_{2n+1:2n+1}|X_{2n+1:2n+1}\leq t)$. If $f$ is increasing on $[0,T]$ we have 
\begin{equation*}
\mathbb E\left[f\left(X_{2n-1:2n-1}\right)|X_{2n-1:2n-1}\leq t\right]\leq\mathbb E\left[f\left(X_{2n+1:2n+1}\right)|X_{2n+1:2n+1}\leq t\right].
\end{equation*}
From the definition of the past extropy it follows that
\begin{eqnarray*}
J\left(_{t}X_{n:n}\right)&=&-\frac{n^2}{2F^{2n}(t)} \int_0^t f^2 (x)F^{2n-2}(x)\mathrm dx\\
&=&\frac{-n^2}{2(2n-1)F(t)}\int_0^t \frac{(2n-1)F^{2n-2}(x)f(x)}{F^{2n-1}(t)}f(x) \mathrm dx\\
&=&\frac{-n^2}{2(2n-1)F(t)}\mathbb E\left[f\left(X_{2n-1:2n-1}\right)|X_{2n-1:2n-1}\leq t\right].
\end{eqnarray*}
Then it follows that
\begin{eqnarray*}
\frac{J\left(_{t}X_{n:n}\right)}{J\left(_{t}X_{n+1:n+1}\right)}&=&\frac{n^2}{(n+1)^2}\frac{2n-1}{2n+1}\frac{\mathbb E\left[f\left(X_{2n-1:2n-1}\right)|X_{2n-1:2n-1}\leq t\right]}{\mathbb E\left[f\left(X_{2n+1:2n+1}\right)|X_{2n+1:2n+1}\leq t\right]}\\
&\leq&\frac{\mathbb E\left[f\left(X_{2n-1:2n-1}\right)|X_{2n-1:2n-1}\leq t\right]}{\mathbb E\left[f\left(X_{2n+1:2n+1}\right)|X_{2n+1:2n+1}\leq t\right]}\leq1.
\end{eqnarray*}
Since the past extropy of a random variable is non-negative we have $J\left(_{t}X_{n:n}\right)\ge J\left(_{t}X_{n+1:n+1}\right)$ and the proof is completed.
\end{proof}

\begin{example}
\label{ex7}
Let $X$ be a random variable distribuited as a Weibull with two parameters, $X\sim W2(\alpha,\lambda)$, i.e. $f(x)=\lambda\alpha x^{\alpha-1}\exp\left(-\lambda x^{\alpha}\right)$. It can be showed that for $\alpha>1$ this pdf has a maximum point $T=\left(\frac{\alpha-1}{\lambda\alpha}\right)^{\frac{1}{\alpha}}$. Let us consider the case in which $X$ has a Weibull distribution with parameters $\alpha=2$ and $\lambda=1$, $X\sim W2(2,1)$ and so $T=\frac{\sqrt 2}{2}$. The hypothesis of the theorem~\ref{thm7} are satisfied for $t=0.5<T=\frac{\sqrt 2}{2}$. Figure 1 shows that $J\left(_{0.5}X_{n:n}\right)$ is decreasing in $n\in\{1,2,\dots,10\}$. Moreover the result of the theorem~\ref{thm7} does not hold for the smallest order statistic as shown in figure 2.
\begin{figure}[htbp]
\centering
\includegraphics[scale=0.5]{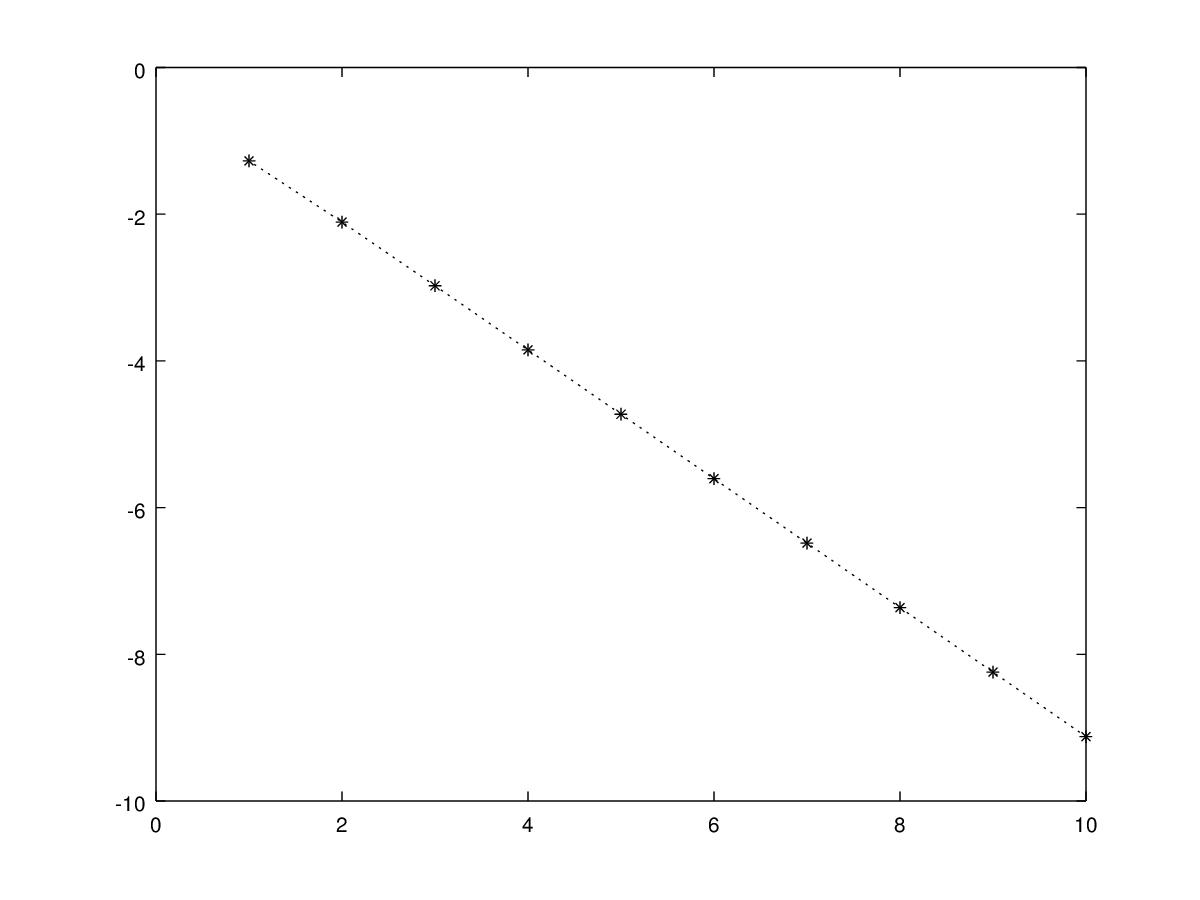}
\caption{$J\left(_{0.5}X_{n:n}\right)$ of a $W2(2,1)$ for $n=1,2,\dots,10$}
\end{figure}
\begin{figure}[htbp]
\centering
\includegraphics[scale=0.5]{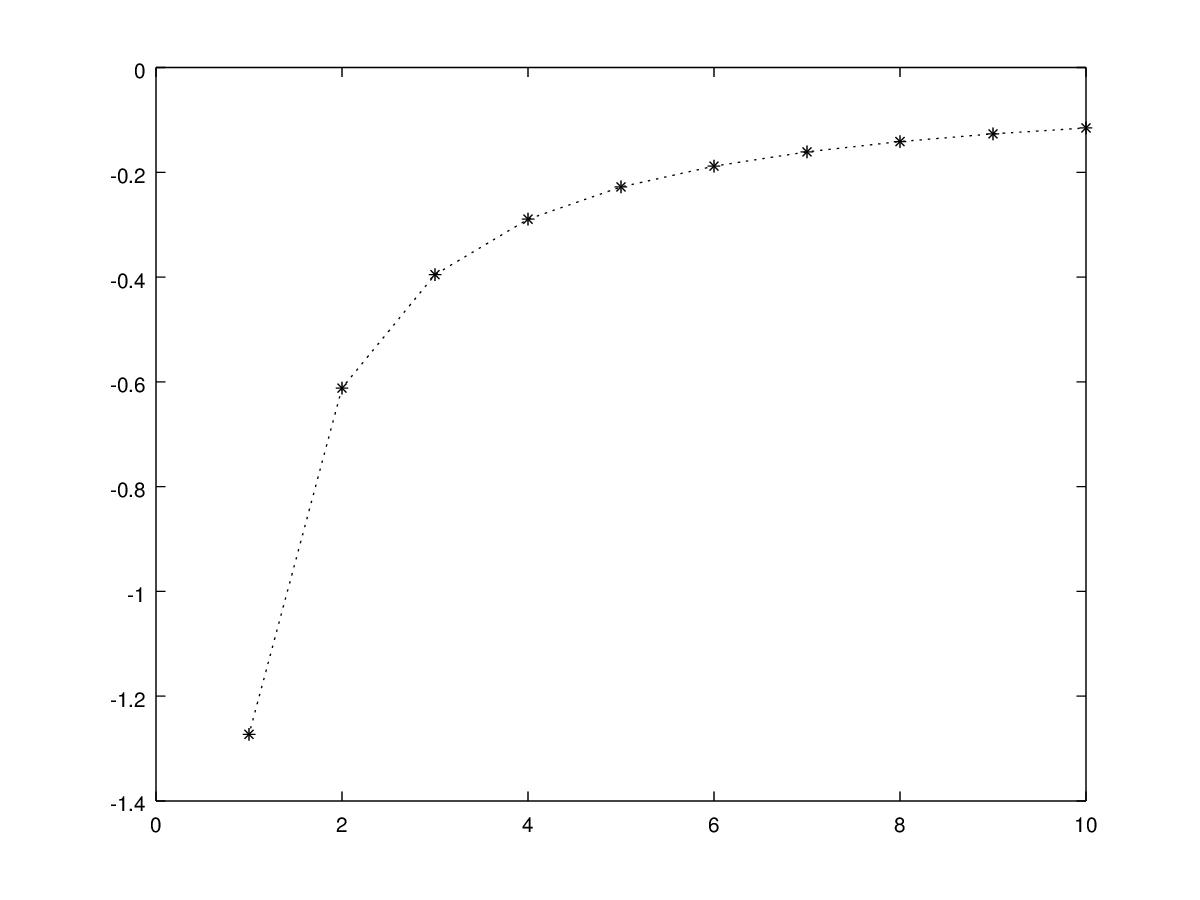}
\caption{$J\left(_{0.5}X_{1:n}\right)$ of a $W2(2,1)$ for $n=1,2,\dots,10$}
\end{figure}

\end{example}

In the case in which $X$ has an increasing pdf on $[0,T]$ with $T>t$ we give a lower bound for $J\left(_{t}X\right)$.

\begin{theorem}
\label{thm8}
If $X$ has an increasing pdf $f$ on $[0,T]$, with $T>t$, then $J\left(_{t}X\right)\ge -\frac{\tau(t)}{2}$.
\end{theorem}

\begin{proof}
From the definition we get
\begin{eqnarray*}
J\left(_{t}X\right)&=&-\frac{1}{2F^2 (t)} \int_0^t f^2 (x)\mathrm dx\\
&=& \frac{-f(t)}{2F(t)}+\frac{1}{2F^2(t)}\int_0^t F(x)f'(x)\mathrm dx\\
&\ge& -\frac{\tau(t)}{2}.
\end{eqnarray*}
\end{proof}

\begin{example}
Let $X\sim W2(2,1)$, as in example~\ref{ex7}, so we know that its pdf is increasing in $[0,T]$ with $T=\frac{\sqrt 2}{2}$. The hypothesis of the theorem~\ref{thm8} are satisfied for $t<T=\frac{\sqrt 2}{2}$. Figure 3 shows that the function $ -\frac{\tau(t)}{2}$ (in red) is a lower bound for the past extropy (in black). We remark that the theorem gives information only for $t\in[0,T]$, in fact for larger values of $t$ the function $ -\frac{\tau(t)}{2}$ could not be a lower bound anymore, as showed in figure 3.
\begin{figure}[htbp]
\centering
\includegraphics[scale=0.6]{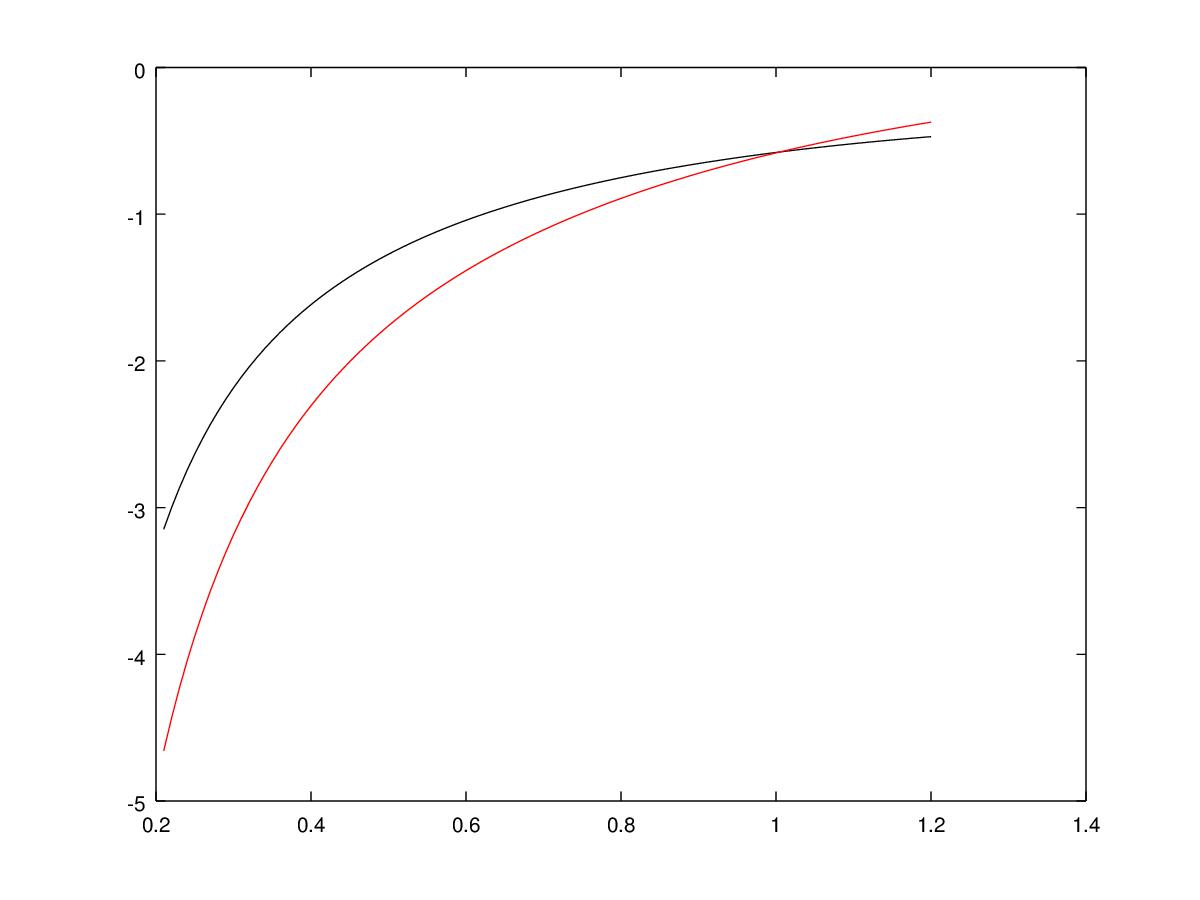}
\caption{$J\left(_{t}X\right)$ (in black) and $-\frac{\tau(t)}{2}$ (in red) of a $W2(2,1)$}
\end{figure}
\end{example}

Qiu (2016), Qiu and Jia (2017) showed that extropy of the $i-th$ order statistics and residual extropy of the $i-th$ order statistics can characterize the underlying distribution uniquely. In the following theorem, whose proof requires next lemma, we show that the past extropy of the largest order statistic can uniquely characterize the underlying distribution.

\begin{lemma}
\label{lem1}
Let $X$ and $Y$ be non-negative random variables such that $J\left(X_{n:n}\right)=J\left(Y_{n:n}\right)$, $\forall n\ge1$. Then $X\overset d=Y$.
\end{lemma}

\begin{proof}
From the definition of the extropy, $J\left(X_{n:n}\right)=J\left(Y_{n:n}\right)$ holds if and only if
\begin{equation*}
\int_0^{+\infty}F_X^{2n-2}(x)f_X^2(x)\mathrm dx=\int_0^{+\infty}F_Y^{2n-2}(x)f_Y^2(x)\mathrm dx
\end{equation*}
i.e. if and only if
\begin{equation*}
\int_0^{+\infty}F_X^{2n-2}(x)\tau_X(x)\mathrm dF^2_X(x)=\int_0^{+\infty}F_Y^{2n-2}(x)\tau_Y(x)\mathrm dF^2_Y(x).
\end{equation*}
Putting $u=F^2_X(x)$ in the left side of the above equation and $u=F^2_Y(x)$ in the right side we have
\begin{equation*}
\int_0^{1}u^{n-1}\tau_X\left(F_X^{-1}(\sqrt u)\right)\mathrm du=\int_0^{1}u^{n-1}\tau_Y\left(F_Y^{-1}(\sqrt u)\right)\mathrm du.
\end{equation*}
that is equivalent to
\begin{equation*}
\int_0^{1}u^{n-1}\left[\tau_X\left(F_X^{-1}(\sqrt u)\right)-\tau_Y\left(F_Y^{-1}(\sqrt u)\right)\right]\mathrm du=0 \mbox{  } \forall n\ge1.
\end{equation*}
Then from Lemma 3.1 of Qui (2017) we get $\tau_X\left(F_X^{-1}(\sqrt u)\right)=\tau_Y\left(F_Y^{-1}(\sqrt u)\right)$ for all $u\in(0,1)$. By taking $\sqrt u=v$ we have $\tau_X\left(F_X^{-1}(v)\right)=\tau_Y\left(F_Y^{-1}(v)\right)$ and so $f_X\left(F_X^{-1}(v)\right)=f_Y\left(F_Y^{-1}(v)\right)$ for all $v\in(0,1)$. This is equivalent to $(F_X^{-1})'(v)=(F_Y^{-1})'(v)$ i.e. $F_X^{-1}(v)=F_Y^{-1}(v)+C$, for all $v\in(0,1)$ with $C$ constant. But for $v=0$ we have $F_X^{-1}(0)=F_Y^{-1}(0)=0$ and so $C=0$.
\end{proof}

\begin{theorem}
Let $X$ and $Y$ be two non-negative random variables with cumulative distribution functions $F(x)$ and $G(x)$, respectively. Then $F$ and $G$ belong to the same family of distributions if and only if for $t\ge0$, $n\ge1$,
\begin{equation*}
J\left(_{t}X_{n:n}\right)=J\left(_{t}Y_{n:n}\right).
\end{equation*}
\end{theorem}

\begin{proof}
It sufficies to prove the sufficiency. $J\left(_{t}X_{n:n}\right)$ is the past extropy for $X_{n:n}$ but it is also the extropy for the variable $_{t}X_{n:n}$. So through lemma~\ref{lem1} we get $_{t}X\overset d=\mbox{ } _{t}Y$. Then $\frac{F(t-x)}{F(t)}=\frac{G(t-x)}{G(t)}$, for $x\in(0,t)$. If exists $t'$ such that $F(t')\ne G(t')$ then in $(0,t')$ $F(x)=\alpha G(x)$ with $\alpha\ne1$. But for all $t>t'$, exists $x\in(0,t)$ such that $t-x=t'$ and so $F(t)\ne G(t)$ and as in the precedent step we have $F(x)=\alpha G(x)$ for $x\in(0,t)$. Letting $t$ to $+\infty$ we have a contradiction because $F$ and $G$ are both distribution function and their limit is 1.
\end{proof}

\section{Conclusion} 
In this paper we studied a measure of uncertainty, the past extropy. It is the extropy of the inactivity time. It is important in the moment in which with an observation we find our system down and we want to investigate about how much time has elapsed after its fail. Moreover we studied some connections with the largest order statistic.

\section{Acknowledgement}
Francesco Buono is partially supported by the GNAMPA research group of INdAM (Istituto Nazionale di Alta Matematica) and MIUR-PRIN 2017, Project "Stochastic Models for Complex Systems" (No. 2017 JFFHSH).

On behalf of all authors, the corresponding author states that there is no conflict of interest.


\begin{thebibliography}{99}
\footnotesize

\bibitem{cover}
Cover, T. M. and Thomas, J. A., 2006. Elements of Information Theory, (2nd edn). Wiley, New York.

\bibitem{dicrescenzo}
Di Crescenzo, A., Longobardi, M., 2002. Entropy-based measure of uncertainty in past lifetime distributions, Journal of Applied Probability, 39, 434--440.

\bibitem{ebrahimi}
Ebrahimi, N., 1996. How to measure uncertainty in the residual life time distribution. Sankhya: The Indian Journal of Statistics, Series A, 58, 48--56.

\bibitem{krishnan}
Krishnan, A. S., Sunoj S. M., Nair N. U., 2020. Some reliability properties of extropy for residual and past lifetime random variables. Journal of the Korean Statistical Society. https://doi.org/10.1007/s42952-019-00023-x.

\bibitem{lad}
Lad, F., Sanfilippo, G., Agrò, G., 2015. Extropy: complementary dual of entropy. Statistical Science 30, 40--58.

\bibitem{qiu}
Qiu, G., 2017. The extropy of order statistics and record values. Statistics \& Probability Letters, 120, 52--60.

\bibitem{qiu2}
Qiu, G., Jia, K., 2018. The residual extropy of order statistics. Statistics \& Probability Letters, 133, 15--22.

\bibitem{shaked}
Shaked, M., Shanthikumar, J. G., 2007. Stochastic orders. Springer Science \& Business Media.

\bibitem{shannon}
Shannon, C. E., 1948. A mathematical theory of communication. Bell System Technical Journal, 27, 379--423. 


\end{thebibliography}
\end{document}